\documentclass[11pt]{amsart}
\usepackage{geometry}               
\usepackage{graphicx}
\usepackage{tikz}
\usepackage{amssymb}
\usepackage{bbm}
\usepackage{pdfsync}
\usepackage{hyperref}
\usepackage{verbatim} 
\usepackage{epstopdf}
\usepackage{color}
\def\R{\mathbb{R}}

\def\N{\mathbb{N}}

\def\Co{\mathbb{C}}

\newcommand{\eps}{\varepsilon}

\renewcommand{\d}{\,{\rm d}}

\newtheorem{theorem}{Theorem}[section]

\newtheorem{definition}[theorem]{Definition}

\newtheorem{proposition}[theorem]{Proposition}
\newtheorem{lemma}[theorem]{Lemma}

\numberwithin{equation}{section}

\title[Block-radial symmetry breaking]{Block-radial symmetry breaking for \\ ground states of
biharmonic NLS}

\author{Rainer Mandel}
\address{
        Karlsruhe Institute of Technology\\
        Institute for Analysis\\
        Englerstrasse 2\\
        D-76131 Karlsruhe, Germany}
\email{rainer.mandel@kit.edu}

\author{Diogo Oliveira e Silva}
\address{ 
Departamento de Matem\'{a}tica\\ 
Instituto Superior T\'{e}cnico\\
Av. Rovisco Pais\\ 
1049-001 Lisboa, Portugal} 
\email{diogo.oliveira.e.silva@tecnico.ulisboa.pt}

\begin{document}

 
\begin{abstract}
We prove that the biharmonic NLS equation
\[\Delta^2 u +2\Delta u+(1+\varepsilon)u=|u|^{p-2}u\,\,\,\textup{ in }\R^d\]
has at least $k+1$ different solutions if  $\eps>0$ is small enough and  $2<p<2_\star^k$, where $2_\star^k$ is 
an explicit critical exponent  arising from the Fourier restriction theory of $O(d-k)\times
O(k)$-symmetric functions. This extends the recent symmetry breaking result of Lenzmann--Weth \cite{LW21}  and
relies on a chain of strict inequalities for the corresponding Rayleigh quotients associated 
with distinct values of $k$. 
We further prove that, as $\eps\to 0^+$, the Fourier transform of each ground state concentrates  near the
unit sphere and becomes rough in the scale of Sobolev spaces.
\end{abstract}

\maketitle

\section{Introduction}

A {\it ground state} for the biharmonic nonlinear Schr\"odinger equation,
\begin{equation}\label{eq_BiHarNLS} 
\Delta^2 u +2\Delta u+(1+\varepsilon)u=|u|^{p-2}u\,\,\,\,\,\textup{ in }\R^d, \,\,\,(p>2,\eps>0)
\end{equation}
 is a nonzero solution $u\in H^2(\R^d)$ of \eqref{eq_BiHarNLS},
at which the infimum 
\begin{equation*}
    R_\varepsilon(p):=\inf_{\textbf{0}\neq u\in H^2(\R^d)}
\frac{{q}_\varepsilon(u)}{\|u\|_p^2},\,\,\,\text{where}\,\,\,{q}_\varepsilon(u):=\int_{\R^d}\left(|\Delta u|^2-2|\nabla u|^2+(1+\varepsilon)|u|^2\right)\,\textup d x,
\end{equation*}
is attained. Lenzmann and Weth  recently established the existence of \textit{nonradial} ground
states $u\in H^2(\R^d)$ of \eqref{eq_BiHarNLS}. More precisely, \cite[Theorem 1.2]{LW21} guarantees the
existence of a threshold $\varepsilon_0>0$ such that every ground state $u\in H^2(\R^d)$ of
\eqref{eq_BiHarNLS} is a nonradial function if $0<\varepsilon<\varepsilon_0$, as long as $d\geq 2$ and
$2<p<2_\star$. Here,  $2_\star:=2\frac{d+1}{d-1}$ is the endpoint Stein--Tomas exponent \cite{St93,To75}
from Fourier restriction theory. 
This symmetry breaking result is especially interesting in view of the recent work by Lenzmann and Sok
\cite{LS2021} on sufficient conditions for radial symmetry of ground states in the general framework of
elliptic pseudodifferential equations.
Our goal is to shed further light on the interplay between symmetries and ground states of elliptic PDEs by
generalizing the symmetry breaking result from~\cite{LW21} to a broader class of (pseudo-)differential
equations and by including the symmetry groups 
$$
  G_k := O(d-k)\times O(k), \qquad k\in\{1,\ldots,d-1\}
$$
in the analysis. A function $f:\R^d\to\Co$ is  \textit{$G_k$-symmetric} if $f\circ
A=f$ holds for every $A\in G_k$. Our main result implies that, for suitable  exponents
$p>2$ and sufficiently small $\eps>0$,  \textit{there exist radial and block-radial
solutions of \eqref{eq_BiHarNLS} whose energy is  larger than that of the ground states}. By our construction,
this immediately implies that \textit{not only do ground states of \eqref{eq_BiHarNLS}  fail to be radial
functions, but they also fail to be block-radial}. The sharp Stein--Tomas exponent $2_\star^k$  for
$G_k$-symmetric functions from~\cite{MOS21} will play a central role. It satisfies 
$2_\star^k=2_\star^{d-k}$ and is explicitly given by
\begin{equation}\label{eq_exp1}
    2_\star^k:=2\frac{d+(d-k)\wedge k}{d-2+(d-k)\wedge k}
    \qquad (x\wedge y:=\min\{x,y\}).
\end{equation}

\medskip
 
We shall consider nonlinear elliptic (pseudo-)differential equations
\begin{equation}\label{eq_BiHarNLSGen}
g_\eps(|D|) u = |u|^{p-2} u\,\,\,\textup{ in }\R^d,
\end{equation}
for a certain class of so-called $(s,\gamma)$-\textit{admissible} symbols $g_\eps$. 
Here, $g_\eps(|D|)u = \mathcal F^{-1}(g_\eps(|\cdot|)\hat u)$. 
In the case $s=\gamma=2$, the symbol $g_\eps(|\xi|)=  (|\xi|^2-1)^2+\varepsilon$
of the biharmonic NLS is a prototypical example.

\begin{definition}\label{def_permissible}
    Let $s>\frac d{d+1}$ and $\gamma>1$. The symbol $g_\varepsilon:\R\to(0,\infty)$ is
    \emph{$(s,\gamma)$-admissible} if it is measurable and  satisfies the  two-sided estimates
\begin{align}
&c(1+|\xi|^{2s}) \leq g_\varepsilon(|\xi|) \leq C(1+|\xi|^{2s}),\,\,\,\textup{ for all }\left||\xi|-1\right|\geq \frac12,\label{eq_A1}\\
&c(\varepsilon+\left||\xi|-1\right|^\gamma)\leq g_\varepsilon(|\xi|) \leq C(\varepsilon+\left||\xi|-1\right|^\gamma),\,\,\,\textup{ for all }\left||\xi|-1\right|\leq \frac12,\label{eq_A2}
\end{align}
for some constants $0<c<C<\infty$ independent of $\eps\in (0,1)$.
\end{definition} 

Some comments are in order. We assume $\gamma>1$ in order to present a 
complete and unified analysis. Optimal results for  $\gamma\in (0,1]$ require
substantial modifications, which will not be investigated here.
The assumption $s> d/(d+1)$ is equivalent to $2_\star<{2d}/(d-2s)_+$, and thus ensures that the
endpoint Stein--Tomas exponent is Sobolev-subcritical for the problem under investigation. 
In light of \cite{BLSS20}, this ensures the existence of ground states for
\eqref{eq_BiHarNLSGen}, i.e., nonzero solutions which attain the infimum
\begin{equation}\label{eq_RayOr}
    \textup R^\circ_\varepsilon(p):=\inf_{\textbf{0}\neq u\in H^s(\R^d) }
    \frac{\textup{q}_\varepsilon(u)}{\|u\|_p^2},
    \quad\text{where }
    \textup{q}_\varepsilon(u):=\int_{\R^d}  |\widehat u (\xi)|^2 g_\varepsilon(|\xi|)\,\textup d \xi.
\end{equation}
For any fixed $\eps>0$, the functional $\textup{q}_\eps$ is positive and continuous on
$H^s(\R^d)$ as long as the symbol $g_\varepsilon$ is $(s,\gamma)$-admissible. More precisely, $u\mapsto
\sqrt{\textup{q}_\eps(u)}$ is then an equivalent norm on $H^s(\R^d)$. Being interested in symmetric solutions
of~\eqref{eq_BiHarNLSGen}, we define the following radial and $G_k$-symmetric versions of the Rayleigh
quotient in \eqref{eq_RayOr}:
 \begin{equation*}
\textup R_\varepsilon^{\textup{rad}}(p):=
  \inf_{\textbf{0}\neq u\in H_{\textup{rad}}^s}
  \frac{\textup{q}_\varepsilon(u)}{\|u\|_p^2};\,\,\,\,\,\,\textup
  R_\varepsilon^k(p):=\inf_{\textbf{0}\neq u\in H_k^s} \frac{\textup{q}_\varepsilon(u)}{\|u\|_p^2},
 \end{equation*}
 where the infima are taken over the Hilbert spaces  $H_{\textup{rad}}^s:=\{f\in H^s(\R^d): f \text{ is
 radial} \}$ and  $H^s_k:=\{f\in H^s(\R^d): f \text{ is } G_k\text{-symmetric}\}$, respectively. 
 We have $\textup R_\varepsilon^k(p)= \textup R_\varepsilon^{d-k}(p)$  with
 identical minimizers up to a trivial change of coordinates. Hence  in our main result we focus on 
 $k\in\{1,\ldots,\lfloor d/2\rfloor\}$.

 \begin{theorem}\label{thm_gen}
 Assume $d\geq 2$, $k\in\{1,\ldots,\lfloor d/2\rfloor\}$,  $2<p<2_\star^k$.
 Let $g_\eps$ be $(s,\gamma)$-admissible  for some  $s>\frac d{d+1}$ and $\gamma>1$.
 Then there exists $\varepsilon_0=\varepsilon_0(p,d,k,s,\gamma)$ such that, for $0<\varepsilon<
 \varepsilon_0$, the chain of strict inequalities
  \begin{equation}\label{eq_chain}
  \textup R^\circ_\varepsilon(p)\vee \textup R_\varepsilon^1(p)<\textup R_\varepsilon^2(p)<\cdots
  <\textup R_\varepsilon^k(p)
  <\textup R_\varepsilon^{k+1}(p) \wedge \ldots\wedge
  \textup R_\varepsilon^{\lfloor d/2\rfloor} (p) 
  \leq \textup R_\varepsilon^{\textup{rad}}(p)
  <\infty
  \end{equation}
  holds and each of these Rayleigh quotients, except possibly $\textup R_\eps^1(p)$, is attained at some
  nontrivial solution of~\eqref{eq_BiHarNLSGen}. In particular, there exist $k+1$ different nontrivial solutions
  and the ground state is neither radial nor block-radial.
\end{theorem}

 Specializing to $g_\eps(|\xi|)=(|\xi|^2-1)^2+\eps$ and  $(s,\gamma,k)=(2,2,1)$,
 we recover the symmetry breaking result from~\cite[Theorem 1.2]{LW21} since
 $2_\star^1 = 2\frac{d+1}{d-1}=2_\star$.
The strict inequalities in~\eqref{eq_chain} rely on lower and upper bounds for the Rayleigh quotients that allow us
to determine the asymptotic behaviour of $\textup R^\ast_\varepsilon(p)$ as $\eps\to 0^+$, for
$\ast\in\{\circ,k,\textup{rad}\}$.
This is the content of our next result, Theorem~\ref{thm_LUBounds}. We introduce the  interpolation parameter
\begin{equation}\label{eq_alphak}
\alpha_k:=\frac{\frac12-\frac1p}{\frac12-\frac1{2_\star^k}}= \left\{ \begin{array}{ll}
(d+k)(\frac12-\frac1p), & \textrm{if $k\leq\frac d2$},\\
(2d-k)(\frac12-\frac1p), & \textrm{if $k>\frac d2$},
\end{array} \right. 
\end{equation}
which satisfies $\alpha_k\in(0,1)$ if and only if $2<p<2_\star^k$.
We further define
\begin{equation}\label{eq_exp2}
    \alpha_{\textup{rad}}:=\frac{\frac12-\frac1p}{\frac12-\frac1{2_\star^{\textup{rad}}}} 
    = 2d\left(\frac{1}{2}-\frac{1}{p}\right), \text{ where }
    2_\star^{\textup{rad}}:=\frac{2d}{d-1},
\end{equation}
which satisfies $\alpha_{\textup{rad}}\in(0,1)$ if and only if $2<p<2_\star^{\textup{rad}}$. 
We also need the Sobolev exponent\footnote{We take this opportunity to record that the exponents defined in \eqref{eq_exp1}, \eqref{eq_exp2} and \eqref{eq_exp3} satisfy
\[2<2_\star^{\textup{rad}}<2_\star^k\leq 2_\star^1=2_\star<2_s^\star\leq\infty.\]}
\begin{equation}\label{eq_exp3}
    2_s^\star:=\left\{ \begin{array}{ll}
\frac{2d}{d-2s}, & \textrm{if $0\leq s<\frac d2$},\\
\infty, & \textrm{if $s\geq\frac d2$},
\end{array} \right. 
\end{equation}
which ensures the largest possible range of validity in the next result.

\begin{theorem}[Lower \& upper bounds]\label{thm_LUBounds}
Assume $d\geq 2$, $k\in\{1,\ldots,d-1\}$,  $2< p< 2_s^\star$.  
Let $g_\varepsilon$ be $(s,\gamma)$-admissible for some $s>\frac{d}{d+1}$ and $\gamma>1$. 
As $\eps\to 0^+$, it holds that 
\begin{equation}\label{eq_LBReps2} 
 \textup R^\circ_\varepsilon(p) 
 \cong \eps^{1-\frac{1\wedge \alpha_1}\gamma}, 
  \end{equation}
\begin{equation}\label{eq_LBReps} 
\textup R_\varepsilon^k(p) \cong \eps^{1-\frac{1\wedge \alpha_k}\gamma},
\end{equation}
\begin{equation}\label{eq_LBReps3} 
\textup R_\varepsilon^{\textup{rad}}(p) \cong
 \left\{ \begin{array}{ll}
 \eps^{1-\frac{1\wedge \alpha_{\textup{rad}}}\gamma},
 & \textrm{if $p\neq 2_\star^{\textup{rad}}$},\\
\eps^{1-\frac1{\gamma}} |\log(\eps)|^{\frac{1-d}d},
&\textrm{if $p= 2_\star^{\textup{rad}}$}.
\end{array} \right. 
\end{equation}
\end{theorem}
\noindent 

Here, $A_\eps\cong B_\eps$ means that there exist constants $0<c<C<\infty$, independent of $\eps>0$, such that
$cB_\eps\leq A_\eps\leq CB_\eps$ holds for all sufficiently small $\eps>0$.
We remark that the logarithmic term appearing when
$p=2_\star^{\textup{rad}}$ is new even in the special case of the biharmonic NLS, and thus refines the
estimates from~\cite{LW21}. This particular bound relies on a
simple but nonstandard interpolation result that the interested reader may find in
Appendix~\ref{app_interpolation}.

\medskip

The proof of Theorem \ref{thm_LUBounds} naturally splits into two parts: lower and upper bounds for each of
the Rayleigh quotients. 
The lower bounds hinge on the following recent $G_k$-symmetric refinement of the Stein--Tomas inequality. If
$d\geq 2$ and $k\in\{1,\ldots,d-1\}$, then \cite[Theorem 1.1]{MOS21} ensures the validity of the adjoint restriction
inequality
\begin{equation}\label{eq_GkSymST}
\|\widehat{f\sigma}\|_{L^p(\R^d)} \leq C(k,d,p) \|f\|_{L^2(\mathbb S^{d-1})} \text{ for all } p\geq 2_\star^k,
\end{equation}
for every $G_k$-symmetric function $f:\R^d\to\Co$, where the adjoint restriction (or \textit{extension}) operator to the unit sphere $\mathbb S^{d-1}:=\{\omega\in\R^d: |\omega|=1\}$ is given by
\begin{equation}\label{eq_ExtOp}
\widehat{f\sigma} (x) = \int_{\mathbb S^{d-1}} e^{i \omega\cdot x} f(\omega) \,\d\sigma(\omega).\,\,\,(x\in\R^d)
\end{equation}
More precisely, estimate~\eqref{eq_GkSymST} 
was shown in~\cite{MOS21} for $k\in\{2,\ldots,d-2\}$, whereas the case $k\in\{1,d-1\}$ is a direct consequence of the
classical Stein--Tomas inequality since $2_\star^1=2_\star^{d-1}=2_\star$. 
That the range of exponents in \eqref{eq_GkSymST} is optimal for $L^2$-densities follows from \cite[Theorem
1.3]{MOS21} and is based on a careful analysis of a $G_k$-symmetric version of Knapp's well-known construction.
Inequality \eqref{eq_GkSymST} and these $G_k$-symmetric Knapp-type constructions together pave the way
towards the precise asymptotics given by Theorem~\ref{thm_LUBounds}.
The asymptotics from Theorem \ref{thm_LUBounds} imply the chain of inequalities \eqref{eq_chain} stated in Theorem \ref{thm_gen}. 
To finish the proof of the latter, one still has to establish the existence of minimizers within the relevant
class of functions. In turn, this amounts to a straightforward argument relying on certain compact embeddings of $H^s_k$,
$k\in\{2,\ldots,d-2\}$, that we recall in Appendix~\ref{app_existence}. If  
$k\in\{1,d-1\}$, then the existence of a minimizer remains an open question, which we discuss in
Appendix~\ref{app_exceptionalcases}.

\medskip

Once the existence of  minimizers has been established, their qualitative properties become of  interest.
To explore this matter, define the interval
\begin{equation}\label{eq_Iinterval}
    I_{\eps,\delta}:=[\delta\eps^{1/\gamma},\delta^{-1}\eps^{1/\gamma}], \,\,\,(\eps,\delta>0)
\end{equation}
and the associated spherical shell
\begin{equation}\label{eq_sphshell}
    A_{\eps,\delta}:=\{\xi\in\R^d: ||\xi|-1|\in I_{\eps,\delta}\}.
\end{equation}
Given a minimizer $u_\eps$ for $\textup R^\ast_\varepsilon(p)$ with
$\ast\in\{\circ,k,\textup{rad}\}$, we split
\begin{equation}\label{eq_split}
u_\eps=v_\eps+w_\eps, \text{ where } \widehat{v}_\eps=\textbf{1}_{A_{\eps,\delta}}\widehat{u}_\eps.
\end{equation}
This decomposition preserves radiality and $G_k$-symmetry.
Our next result reveals that $u_\eps$ {\it concentrates on the unit sphere in Fourier space}, as $\eps\to 0^+$.

\begin{theorem}\label{thm_concentration}
Assume $d\geq 2$, $k\in\{1,\ldots,d-1\}$, $2< p< 2_s^\star$. 
Let $g_\varepsilon$ be $(s,\gamma)$-admissible for some $s>\frac{d}{d+1}$ and $\gamma>1$.
Let $\ast\in\{\circ,k,\textup{rad}\}$.
If $u_\eps$ is a minimizer for $\textup R^\ast_\varepsilon(p)$,
decomposed as in \eqref{eq_split} with $\delta=\delta_\eps$ for a given positive null sequence
$(\delta_\eps)_{\eps>0}$, then \[\lim_{\eps\to 0^+} \frac{\|w_\eps\|_p}{\|v_\eps\|_p}
=\lim_{\eps\to 0^+} \frac{\textup{q}_\eps(w_\eps)}{\textup{q}_\eps(v_\eps)}=0.\]
\end{theorem}

We emphasize that this result also applies to the case $\ast = k\in\{1,d-1\}$ 
where the existence of a minimizer is still open. 

\medskip

In the smaller
$G_k$-symmetry breaking range $2<p<2_\star^k$,  minimizers   \textit{become rough in Fourier space}, as $\eps\to 0^+$. 
The intuition comes from the fact that 
the upper bounds in Theorem \ref{thm_LUBounds} were proved via test functions that resemble Knapp
counterexamples. The Fourier transform of such functions vary sharply  along
spheres. Our next result indicates that such behaviour is somewhat necessary in order to be energetically efficient. The regularity is measured in the Sobolev space $H^t(\mathbb S^{d-1})$ of
functions having $t\geq 0$ derivatives in $L^2(\mathbb S^{d-1})$. This space is defined via spherical
harmonic expansions, e.g.\ as in \cite[\S1.7.3, Remark 7.6]{LM72}, or equivalently by considering a smooth
finite partition of unity and diffeomorphisms onto the unit ball of $\R^{d-1}$ together with the usual Sobolev
norm on $\R^{d-1}$.

\begin{theorem}\label{thm_roughness}
Assume $d\geq 2$, $k\in\{1,\ldots,d-1\}$,  $2< p< 2_\star^k$.
Let $g_\varepsilon$ be $(s,\gamma)$-admissible for some $s>\frac{d}{d+1}$ and $\gamma>1$. Let
$\ast\in\{\circ,k\}$.
If $u_\eps$ is a minimizer for $\textup R^\ast_\varepsilon(p)$, then  for every $t>0$ and
every positive null sequence $(\delta_\eps)_{\eps>0}$,  it holds that
\begin{equation}\label{eq_ConclThmRough}
    \lim_{\eps\to 0^+}\sup_{|r-1|\in I_{\eps,\delta_\eps}} \frac{\|\widehat{u}_\eps(r\,\cdot)\|_{H^t(\mathbb S^{d-1})}}{\|\widehat{u}_\eps(r\,\cdot)\|_{L^2(\mathbb S^{d-1})}} = \infty.
\end{equation}
\end{theorem}
Given that radially symmetric functions are constant on spheres around the origin, this phenomenon cannot occur in the
radial case.

\subsection{Structure of the paper}
We prove Theorem \ref{thm_LUBounds} in \S\ref{sec:LowerUppper}, and this implies the first part of Theorem \ref{thm_gen}.
We prove Theorem \ref{thm_concentration} in \S\ref{sec_concentration},
and Theorem \ref{thm_roughness} in \S\ref{sec_roughness}.
We recall the classical proof of the second part of Theorem \ref{thm_gen} in Appendix \ref{app_existence} and
establish a technical interpolation result that is needed for the proof of Theorem \ref{thm_LUBounds} in
Appendix \ref{app_interpolation}. Finally,  Appendix \ref{app_exceptionalcases} contains some further considerations regarding the exceptional cases $k\in\{1,d-1\}$.

\subsection{Notation}
We use $X\lesssim Y$ or $Y\gtrsim X$ to denote the estimate $|X|\leq CY$ for an absolute positive constant $C$,  and $X\cong Y$ to denote the estimates $X\lesssim Y \lesssim X$.
We  occasionally allow the implied constant $C$ in the above notation to depend on additional parameters,
which we will indicate by subscripts (unless explicitly omitted); thus for instance $X \lesssim_j Y$ denotes
an estimate of the form $|X| \leq C_jY$ for some $C_j$ depending on $j$.
The indicator function of a set $E$ is denoted by $\textbf{1}_E.$

\section{Lower and upper bounds}\label{sec:LowerUppper}

In this section, we prove Theorem \ref{thm_LUBounds}. The following simple result will be useful for both  
lower and   upper bounds.

\begin{lemma}\label{lem_technical}
Let $\eps\in(0,1)$ and let $g_\eps$ be $(s,\gamma)$-admissible for some $s>\frac{d}{d+1}$ and $\gamma>1$.
 Then
\[\int_{\frac12}^{\frac32}\frac{\textup d r}{g_\eps(r)} \cong \eps^{\frac1{\gamma}-1}.\]
\end{lemma}
\begin{proof}
This follows from a simple change of variables and $\gamma>1$:
\[\int_{\frac12}^{\frac32}\frac{\textup d r}{g_\eps(r)}
\stackrel{\eqref{eq_A2}}\cong \int_{\frac12}^{\frac32}\frac{\textup d r}{\eps+|r-1|^\gamma}
=2\eps^{\frac1{\gamma}-1} \int_0^{\frac1{2\eps^{1/\gamma}}} \frac{\d t}{1+t^\gamma}\cong\eps^{\frac1{\gamma}-1}.\qedhere\]
\end{proof}

\subsection{Lower bounds}\label{sec:lower}

We start by addressing the lower bound in \eqref{eq_LBReps}. We decompose an arbitrary $G_k$-symmetric
function $u\in H_k^s$ as
\begin{equation}\label{eq_DecUVW}
u=v+w, \text{ where } \widehat v(\xi)=\textbf{1}_{||\xi|-1|\leq \frac{1}{2}}  \widehat u(\xi) \text{ and }
\widehat w(\xi)=\textbf{1}_{||\xi|-1|>\frac{1}{2}}\widehat u(\xi),
\end{equation}
and observe that $v,w$ are still $G_k$-symmetric.
We then have $\|u\|_p\leq\|v\|_p+\|w\|_p$ and $\textup{q}_\eps(u)= \textup{q}_\eps(v)+\textup{q}_\eps(w)$. It
follows  that
\begin{equation*}
    \frac{\textup{q}_\eps(u)}{\|u\|_p^2}
    \geq\frac{\textup{q}_\eps(v)+\textup{q}_\eps(w)}{(\|v\|_p+\|w\|_p)^2}\geq\frac12\frac{\textup{q}_\eps(v)+\textup{q}_\eps(w)}{\|v\|_p^2+\|w\|_p^2}\geq\frac12 \left(\frac{\textup{q}_\eps(v)}{\|v\|_p^2}\wedge\frac{\textup{q}_\eps(w)}{\|w\|_p^2}\right).
\end{equation*}
Hence it suffices to prove lower bounds for the quotients corresponding to $v$ and $w$ separately. Since $2<
p< 2_s^\star$, Sobolev embedding ensures $H^s\subset L^2\cap L^{2_s^\star}\subset L^p$, and thus
\begin{equation} \label{eq_lowerbound_w}
\|w\|_p^2\lesssim \|w\|_{H^s}^2=\int_{\R^d} |\widehat w(\xi)|^2 (1+|\xi|^2)^s\,\textup d
\xi\cong\int_{\R^d} |\widehat w(\xi)|^2 g_\eps(|\xi|)\,\textup d \xi = \textup q_\eps(w),
\end{equation}
where the $\cong$-estimate holds uniformly with respect to small $\eps>0$ thanks to \eqref{eq_A1}.
The lower bound for $\|w\|_p^{-2} \textup q_\eps(w)$ follows at once, and so we focus on $v$.

\medskip

As in~\cite{LW21} the proof of lower bounds for $\|v\|_p^{-2} \textup q_\eps(v)$  splits into two cases,
according to whether or not $p$ is larger than $2_\star^k$.
If $p\geq 2_\star^k$, then we may use the $G_k$-symmetric Stein--Tomas inequality \eqref{eq_GkSymST} as
follows:
\begin{align*}
    &\|v\|_p^2
    \cong \left\| \int_{\R^d} e^{i\langle\xi,\cdot\rangle} \widehat v(\xi)\,\d\xi\right\|_p^2
    \leq  \left(\int_{\frac12}^{\frac32}  \left\|\int_{\mathbb S^{d-1}} e^{i\langle\omega,r\cdot\rangle} \widehat v(r\omega)\,\textup d\sigma(\omega)\right\|_p  \,\d r\right)^2\label{eq_bigchain}\\
    &\cong\left(\int_{\frac12}^{\frac32} \left\|\int_{\mathbb S^{d-1}} e^{i\langle\omega,\cdot\rangle}
    \widehat v(r\omega)\,\textup d\sigma(\omega)\right\|_p  \,\d r\right)^2 \lesssim \left(\int_{\frac12}^{\frac32} \|\widehat v(r\cdot)\|_{L^2(\mathbb S^{d-1})} \,\d r\right)^2\notag\\
    &= \left(\int_{\frac12}^{\frac32}  \|\widehat v(r\cdot)\|_{L^2(\mathbb S^{d-1})}^2g_\eps(r) \,\d r
    \right)\left(\int_{\frac12}^{\frac32}  \frac{\d r}{g_\eps(r)}\right)\lesssim \left( \int_{\R^d}  |\widehat v(\xi)|^2 g_\eps(|\xi|)\, \d \xi \right) \left(\int_{\frac12}^{\frac32} \frac{\d r}{g_\eps(r)}\right).\notag
\end{align*}
By Lemma \ref{lem_technical}, we then conclude
\begin{equation}\label{eq_UBsuperP}
\|v\|_p^2 \lesssim\textup{q}_\eps(v) \int_{\frac12}^{\frac32}\frac{\d r}{g_\eps(r)}\cong\eps^{\frac1{\gamma}-1}\textup{q}_\eps(v),
\end{equation}
which implies the claimed lower bound. If $2<p<2_\star^k$, then the interpolation parameter from
\eqref{eq_alphak} satisfies $\alpha_k\in(0,1)$, and $\|v\|_p \leq \|v\|_2^{1-\alpha_k}
\|v\|_{2_\star^k}^{\alpha_k}$. Plancherel's identity, the Fourier support assumption on $v$ and the two-sided
estimate \eqref{eq_A2} lead to
\begin{equation}\label{eq_L2BD}
\|v\|_2^2=\int_{\R^d} |\widehat v(\xi)|^2\,\textup d \xi \lesssim \eps^{-1} \int_{\R^d} |\widehat v(\xi)|^2 g_\eps(|\xi|)\,\textup d \xi =\eps^{-1} \textup q_\eps(v).
\end{equation}
From \eqref{eq_UBsuperP} with $p=2_\star^k$ and \eqref{eq_L2BD}, it then follows that
\begin{equation}\label{eq_UBvInterpol}
    \|v\|_p^2\leq  (\|v\|_2^2)^{1-\alpha_k} (\|v\|_{2_\star^k}^2)^{\alpha_k} \lesssim (\eps^{-1} \textup q_\eps(v))^{1-\alpha_k} (\eps^{\frac1{\gamma}-1}\textup{q}_\eps(v))^{\alpha_k}=\eps^{\frac{\alpha_k}{\gamma}-1}\textup{q}_\eps(v).
\end{equation}
This concludes the proof of \eqref{eq_LBReps}.
The lower bounds in \eqref{eq_LBReps2} are proved analogously given that
it suffices to replace the exponent $2_\star^k$ originating from the $G_k$-symmetric Stein--Tomas
inequality by the exponent $2_\star=2_\star^1$ coming from the classical Stein--Tomas inequality. 
 We omit the details.
(See~\cite[\S4]{LW21} for a proof in the special case $s=\gamma=2$.)

\medskip

We now  verify the lower bound in \eqref{eq_LBReps3}.
Using the decomposition \eqref{eq_DecUVW} and reasoning as in \eqref{eq_lowerbound_w}, it suffices to
consider  a radial function $v\in H_{\textup{rad}}^s$ whose Fourier support is contained in the spherical
shell $\{\xi\in\R^d: \frac12\leq |\xi| \leq \frac32\}$. By Fourier inversion, \[v(x)=_d  |x|^{\frac {2-d}2}
\int_{\frac12}^{\frac32} \widehat v(r)J_{\frac {d-2}2}(r|x|) r^{\frac d2}\,\textup d r,\] where $\widehat v(r)=\widehat v(r\omega)$ for $r>0$ and almost every $\omega\in\mathbb{S}^{d-1}$. The appearance of the
Bessel function is due to
\begin{equation}\label{eq_sigmahat}
    \widehat\sigma(\xi)=_d |\xi|^{\frac {2-d}2} J_{\frac{d-2}2}(|\xi|).
\end{equation}
Standard asymptotics for Bessel functions at zero and at infinity lead to the pointwise estimate
\begin{align*}
|v(x)|
&\lesssim_d (1+|x|)^{\frac{1-d}2} \int_{\frac12}^{\frac32} |\widehat v(r)|\,\d r \\
&\leq (1+|x|)^{\frac{1-d}2} \left( \int_{\frac12}^{\frac32} \frac{\d r}{\eps+|r-1|^\gamma}\right)^{\frac12}
 \left( \int_{\frac12}^{\frac32} (\eps+|r-1|^\gamma) |\widehat v(r)|^2 \d r \right)^{\frac12}\\
&\cong (1+|x|)^{\frac{1-d}2} \eps^{\frac{1}{2}(\frac{1}{\gamma}-1)} \sqrt{\textup
q_\eps(v)},
\end{align*}
where we used the Cauchy--Schwarz inequality,  Lemma \ref{lem_technical}, and assumption \eqref{eq_A2} on $g_\eps$. 
Recall that $2_\star^{\textup{rad}}=\frac{2d}{d-1}$. We thus infer
\begin{equation}\label{eq_ToInterpolate}
\|v\|_p+\|v\|_{2_\star^{\textup{rad}},\infty} \lesssim \eps^{\frac{1}{2}(\frac{1}{\gamma}-1)}\sqrt{\textup q_\eps(v)},\text{ for every }p>2_\star^{\textup{rad}},
\end{equation}
which implies the lower bound \eqref{eq_LBReps3} in the range $p>2_\star^{\textup{rad}}$.
If $2<p<2_\star^{\textup{rad}}$, then real interpolation between \eqref{eq_ToInterpolate} and the simple $L^2$-bound
$\|v\|_2^2\lesssim \eps^{-1} \textup q_\eps(v)$ from \eqref{eq_L2BD}
yields
\[\|v\|_p^2
\lesssim (\|v\|_2^2)^{{1-\alpha_{\textup{rad}}}} (\|v\|_{2_\star^{\textup{rad}},\infty}^2)^{\alpha_{\textup{rad}}}
\lesssim \eps^{\frac{\alpha_{\textup{rad}}}{\gamma}-1}\textup q_\eps(v);\]
see \cite[Prop.\@ 1.1.14]{Gr14}.
Finally, the critical case $p=2_\star^{\textup{rad}}$ is a consequence of  Proposition \ref{prop:interpolate} with parameters
\[(r,q,C_1,C_2)=\left(2,2_\star^{\textup{rad}},(\eps^{-1}\textup
q_\eps(v))^{\frac12},(\eps^{\frac1{\gamma}-1}\textup q_\eps(v))^{\frac12}\right),\]
as follows: if $\eps>0$  is sufficiently small, then
$$
  \|v\|_{2_\star^{\textup{rad}}}^2
  \lesssim C_2^2 (1+\log_+(C_1/C_2))^{\frac{2}{2_\star^{\textup{rad}}}}
  \cong  \eps^{\frac1{\gamma}-1} |\log(\eps)|^{\frac{d-1}{d}} \textup q_\eps(v).
$$
This concludes the verification of the lower bounds in Theorem \ref{thm_LUBounds}.

\subsection{Upper bounds}\label{sec:upper}

We first consider the upper bound in \eqref{eq_LBReps} for $G_k$-symmetric functions.
Given sufficiently small $\eps,\delta>0$, define the $G_k$-symmetric trial function $v_{\eps,\delta}$ via its Fourier transform as
\begin{equation}\label{eq_FunctionV}
\widehat{v}_{\eps,\delta}(\xi) := 
\textbf{1}_{||\xi|-1|\leq m_{\eps,\delta}} \cdot \textbf{1}_{\mathcal
C_\delta^k}\left(|\xi|^{-1}{\xi}\right) \cdot a_\eps(||\xi|-1|).
\end{equation}
Here, both $m_{\eps,\delta}>0$ and the profile function $a_\eps$ will be chosen below, and $\mathcal
C_\delta^k\subset\mathbb S^{d-1}$ is a union of two spherical caps of radius $\delta$ defined as follows:
\[\mathcal C_\delta^k:=\{(\eta,\zeta)\in\R^{d-k}\times\R^{k}: |\eta|^2+|\zeta|^2=1, |\eta|<\delta\}.\]
The set $\mathcal C_\delta^k$ is $G_k$-symmetric and has surface area $\sigma(\mathcal C_\delta^k)=_d
\delta^{d-k}.$ Our estimates for $v_{\eps,\delta}$ make use of the following auxiliary result.

\begin{proposition} \label{prop:pointwiselowerbound}
  Let $d\geq 2$ and $k\in\{1,\ldots,d-1\}$. 
  There exist $c_0,c_1,c_2>0$ with the following property: for
  every $(m,\delta)\in (0,\frac{1}{2})^2$,   
 there exist disjoint measurable sets $E_j\subset\R^d, j\in\{1,\ldots, \lfloor
  \frac{c_0}{\delta^2+m}\rfloor\}$, satisfying $|E_j|\geq c_1\delta^{k-d} j^{k-1}$ and 
  $$
    \int_{\mathbb S^{d-1}}  e^{irx\cdot\omega} {\bf 1}_{\mathcal
    C_\delta^k}(\omega)\,\d\sigma(\omega) 
    \geq c_2 \delta^{d-k} j^{\frac{1-k}{2}},
  $$
 whenever  $|r-1|\leq m$ and $x\in E_j$.
\end{proposition}
\begin{proof}
  If $k\in\{2,\ldots,d-2\}$, then this result is a variant of the computations from the proof of
  \cite[Theorem 1.3]{MOS21}, which we recall for the reader's convenience.  The starting point is the
  formula 
  \begin{align} \label{eq:Ck}
    \begin{aligned}
    \widehat{\textbf{1}_{\mathcal C^k_\delta}\sigma}(rx)  \cong 
    \int_0^\delta &\rho^{d-k-1} (1-\rho^2)^{\frac{k-2}2}\times \\
    &  (\rho r|y|)^{\frac{2-d+k}2}J_{\frac{d-k-2}2}(\rho r|y|) \cdot (\sqrt{1-\rho^2}r|z|)^{\frac{2-k}2}
    J_{\frac{k-2}2}(\sqrt{1-\rho^2}r|z|)\d \rho
    \end{aligned}
  \end{align}
  where $x=(y,z)\in\R^{d-k}\times\R^k$ and $|r-1|\leq m$; see~\cite[Eq.~(6.3)]{MOS21}.
  There  this quantity is  estimated from below by $c_1 \delta^{d-k} j^{\frac{1-k}{2}}$, provided 
  $x=(y,z)$ belongs to the set 
  \[ 
    \tilde E_{j}(r):=\left\{(y,z)\in\R^{d-k}\times\R^k: 0\leq r|y|\leq
    c\delta^{-1},\frac{z_j-c}{\sqrt{1-\delta^2}}\leq r|z|\leq z_j+c\right\}
  \]
  for some sufficiently small absolute constant $c>0$ satisfying  $c<\frac14\inf\{z_{j+1}-z_j: j\in\N\}$.
  Here, $\{z_j\}_{j\geq 1}$ denotes the increasing sequence of local maxima of the Bessel function $J_{(k-2)/2}$,
  which satisfies $z_j=2\pi j+O(1)$ as $j\to\infty$.
   For every $r\in [1-m,1+m]$ and $m\in(0,\frac12)$, we have $E_j\subset\tilde E_j(r)$, where  
  \[ 
    E_{j}:=\left\{(y,z)\in\R^{d-k}\times\R^k: 0\leq |y|\leq
    \frac{c}{\delta (1+m)},\frac{z_j-c}{\sqrt{1-\delta^2}(1-m)}\leq |z|\leq \frac{z_j+c}{1+m}\right
    \}.
  \]
  To bound the measure of $E_j$ from below,  choose 
  $\alpha,\beta>0$ such that the estimates
$$
   \left(\frac{z_j-c}{z_j+c}\right)^k \leq 1-\frac{\alpha}{j},\qquad  
   \left(\frac{1+m}{\sqrt{1-\delta^2}(1-m)}\right)^k
   \leq  1+\beta (\delta^2+m)
$$
hold for $(\delta,m)\in (0,\frac{1}{2})^2$ and all $j\in\N$.  We then find, for $j= 1, \ldots,
\lfloor\frac{c_0}{\delta^2+m}\rfloor$ with $c_0 := \frac{\alpha}{2\beta}$,
\begin{align*}
  |E_j|
  &\gtrsim 
  \left(\frac{c}{\delta(1+m)}\right)^{d-k}   
  \left[ \left(\frac{z_j+c}{1+m}\right)^k  - \left(\frac{z_j-c}{\sqrt{1-\delta}(1-m)}\right)^k
  \right] \\
  &\gtrsim \delta^{k-d} j^k 
  \left[ 1 - \left(\frac{z_j-c}{z_j+c}\right)^k \left(\frac{1+m}{\sqrt{1-\delta^2}(1-m)}\right)^k
  \right] \\
  &\gtrsim \delta^{k-d} j^k 
  \left[ 1 - (1-\frac{\alpha}{j})\cdot (1+\beta (\delta^2+m)) 
  \right] \\
  &\gtrsim  \delta^{k-d} j^k 
  \left[  \frac{\alpha}{j} -  \beta (\delta^2+m) 
  \right] \\
  &\gtrsim \delta^{k-d} j^{k-1}. 
\end{align*}
  This finishes the proof for $k\in\{2,\ldots,d-2\}$. Formula~\eqref{eq:Ck} continues to hold in the case $k\in\{1,d-1\}$ by taking into account the fact that 
  $$
    J_{-\frac{1}{2}}(z) = \sqrt{\frac{2}{\pi z}}  \cos(z) .
  $$ 
  The same proof, with $z_j$ replaced by $2\pi j$, then yields the result.
\end{proof}

\begin{lemma}\label{lem_UpperBoundsPrep}
Assume $2< p<\frac{2k}{k-1}$ and let $a_\eps\in L^2_{\textup{loc}}(\R_+)$ be nonnegative.
Then, for all sufficiently small $\delta,\eps,m_{\eps,\delta}>0$ and $v_{\eps,\delta}$ as defined in
\eqref{eq_FunctionV},
\begin{align}
\textup q_\eps(v_{\eps,\delta})&\cong \delta^{d-k} \int_0^{m_{\eps,\delta}} a_\eps^2(s) (\eps+s^\gamma)\,\d s,\label{eq_qV}\\
\|v_{\eps,\delta}\|_p&\gtrsim \delta^{\frac{d-k}{p'}} (\delta^2+m_{\eps,\delta})^{\frac{k-1}2-\frac kp} \int_0^{m_{\eps,\delta}} a_\eps(s)\,\d s.\label{eq_vLp}
\end{align}
\end{lemma}
\begin{proof}
We alleviate notation by writing $m=m_{\eps,\delta}$ and $v=v_{\eps,\delta}$.
Estimate \eqref{eq_qV} is a simple consequence of first passing to polar coordinates in Fourier
space, and then applying the two-sided inequality \eqref{eq_A2} together with the change of variables $s=r-1$:
\[\textup q_\eps(v)=  \int_{\R^d} |\widehat v(\xi)|^2 g_\eps(|\xi|)\,\d\xi
=\sigma(\mathcal C_\delta^k) \int_{1-m}^{1+m}  a_\eps^2(|r-1|) g_\eps(r)\,\d r
\cong\delta^{d-k}\int_0^m a_\eps^2(s) (\eps+s^\gamma)\,\d s.\]
Estimate \eqref{eq_vLp} stems from the following identity, obtained via Fourier inversion:
\begin{equation}\label{eq_InnerOsc}
v(x)=_d\int_{1-m}^{1+m} a_\eps(|r-1|) \left( \int_{\mathbb S^{d-1}}  e^{irx\cdot\omega} \textbf{1}_{\mathcal
C_\delta^k}(\omega)\,\d\sigma(\omega)\right) r^{d-1}\,\d r.
\end{equation}
Choosing  $c_0>0$ and $E_j\subset\R^d$ as in Proposition~\ref{prop:pointwiselowerbound} we obtain, for
$j\in\{1,\ldots,\lfloor \frac{c_0}{\delta^2+m}\rfloor\}$,
\begin{equation} \label{eq_lowerbound1}
  |v(x)| \gtrsim\delta^{d-k}j^{\frac{1-k}2} \int_0^m a_\eps(s)\,\d s, \text{ for every } x\in E_j.
\end{equation}
Using that $|E_j|\gtrsim \delta^{k-d} j^{k-1}$, we thus obtain for $2< p<\frac{2k}{k-1}$: 
\begin{align*}
    \|v\|_p
    &\stackrel{\eqref{eq_lowerbound1}}\gtrsim  \delta^{d-k}\left(\sum_{j=1}^{\lfloor \frac
    {c_0}{\delta^2+m}\rfloor} (j^{\frac{1-k}2})^p |E_j|\right)^{\frac1p} \int_0^m a_\eps(s)\,\d s\\
    &\gtrsim  \delta^{d-k}\left(\sum_{j=1}^{\lfloor \frac
    {c_0}{\delta^2+m}\rfloor} j^{\frac{1-k}2 (p-2)} \delta^{k-d}\right)^{\frac1p} \int_0^m a_\eps(s)\,\d s\\
    &\cong  \delta^{\frac{d-k}{p'}} (\delta^2+m)^{\frac{k-1}2-\frac kp} \int_0^m a_\eps(s)\,\d s,
\end{align*}
which is all we had to show. Note that we used $p<\frac{2k}{k-1}$ in the last estimate.
\end{proof}

In the radial case, the following substitute holds for the simpler trial function given by
\begin{equation}\label{eq_FunctionW}
  \widehat v_\eps(\xi):= \textbf{1}_{||\xi|-1|\leq m_\eps}  \cdot a_\eps(||\xi|-1|).
\end{equation}

\begin{lemma}\label{lem_radial}
Let $a_\eps\in L^2_{\textup{loc}}(\R_+)$ be nonnegative. Then, for all sufficiently small $\eps,m_{\eps}>0$ and
$v_{\eps}$ as defined in \eqref{eq_FunctionW},
\begin{align}
q_\eps(v_{\eps})&\cong \int_0^{m_{\eps}} a_\eps^2(s) (\eps+s^\gamma)\,\d s,\label{eq_qWEst}\\
\|v_{\eps}\|_p&\gtrsim  \int_0^{m_{\eps}} a_\eps(s)\,\d s\cdot  \left\{ \begin{array}{ll}
 m_{\eps}^{\frac{d-1}2-\frac dp}, & \textrm{if $p\in (2,2_\star^{\textup{rad}})$,}\\
|\log(m_\eps)|^{\frac1p}, & \textrm{if $p=2_\star^{\textup{rad}}$,}\\
1, & \textrm{if $p\in(2_\star^{\textup{rad}},\infty)$.}
\end{array} \right.\label{eq_WpEst}
\end{align}
\end{lemma}

\begin{proof}
    The proof is analogous to that of Lemma \ref{lem_UpperBoundsPrep}, so we just highlight the differences.
    Write $m=m_\eps$ and $v=v_\eps$.
    The lower bound \eqref{eq_WpEst} relies on \eqref{eq_sigmahat} for the representation
    \[v(x) =_d |x|^{\frac{2-d}2}\int_{1-m}^{1+m} a_\eps(|r-1|)  J_{\frac{d-2}2}(r|x|) r^{\frac{d}2}\,\d
    r.\] Let again $z_j=2\pi j+O(1)$, as $j\to\infty$, denote the sequence of local maxima of
    $J_{\frac{d-2}2}$.
    With $c>0$  as in the proof of Proposition \ref{prop:pointwiselowerbound}, it holds that
    \[|J_{\frac{d-2}2}(r|x|)|\gtrsim j^{-1/2},\text{ if }|r-1|\leq m \text{ and }
    \frac{z_j-c}{1-m}\leq |x|\leq \frac{z_j+c}{1+m}.\] For such $x$, we get the pointwise lower bound
    \[|v(x)|\gtrsim j^{\frac{1-d}2} \int_0^m a_\eps(s)\,\d s.\]
    Arguing as above, we find that this lower bound holds for $x$ belonging to annular regions
    $E_j\subset\R^d$ satisfying 
    $$
      |E_j| 
      \gtrsim \left(\frac{z_j+c}{1+m}\right)^d - \left(\frac{z_j-c}{1-m}\right)^d
      \gtrsim j^{d-1}  
    $$
    provided that $j\in\{1,2,\ldots,\lfloor \frac{c_0}m\rfloor\}$ for some suitably small $c_0>0$.
    Integrating these estimates and summing up the resulting bounds yields
    $$
      \|v\|_p
      \gtrsim \left(\sum_{j=1}^{\lfloor
       \frac{c_0}m\rfloor} j^{\frac{p(1-d)}{2}} |E_j| \right)^{\frac{1}{p}}
      \int_0^m a_\eps(s)\,\d s
      \gtrsim \left(\sum_{j=1}^{\lfloor
       \frac{c_0}m\rfloor} j^{\frac{(p-2)(1-d)}{2}} \right)^{\frac{1}{p}}
      \int_0^m a_\eps(s)\,\d s.
    $$
    The three cases in \eqref{eq_WpEst} correspond to the exponent $\frac{(p-2)(1-d)}{2}$ being larger than, equal to, or
    smaller than $-1$.
\end{proof}

\begin{proof}[Proof of Theorem~\ref{thm_LUBounds} (Upper bounds)]
 {\bf Radial case \eqref{eq_LBReps3}.}
By Lemma \ref{lem_radial}, we have
\begin{align*}
    \textup R_\varepsilon^{\textup{rad}}(p)
    \lesssim
    \frac{\int_0^{m_\eps} a_\eps^2(s)(\eps+s^\gamma)\,\d s}{\left(\int_0^{m_\eps} a_\eps(s)\,\d s\right)^2}\times\left\{ \begin{array}{ll}
 m_{\eps}^{1-d+\frac {2d}p}, & \textrm{if $p\in(2,2_\star^{\textup{rad}})$,}\\
|\log(m_\eps)|^{-\frac2p}, & \textrm{if $p=2_\star^{\textup{rad}}$,}\\
1, & \textrm{if $p\in(2_\star^{\textup{rad}},\infty)$.}
\end{array} \right.
\end{align*}
Taking $a_\eps(s)=(\eps+s^\gamma)^{-1}$, we find that the first factor simplifies to
$\left(\int_0^{m_\eps} (\eps+s^\gamma)^{-1}\d s\right)^{-1}$.
Choosing  $m_\eps:=\eps^{\frac1{\gamma}}$, we then have
\begin{align*}
    \textup R_\varepsilon^{\textup{rad}}(p)
    &\lesssim
    \left\{ \begin{array}{ll}
 \eps^{1-\frac{1}{\gamma}} \cdot \eps^{\frac{1}{\gamma}(1-d+\frac {2d}p)}, & \textrm{if
 $p\in(2,2_\star^{\textup{rad}})$,}\\
\eps^{1-\frac{1}{\gamma}} \cdot |\log(\eps)|^{-\frac2p}, & \textrm{if $p=2_\star^{\textup{rad}}$,}\\
\eps^{1-\frac{1}{\gamma}}, & \textrm{if $p\in(2_\star^{\textup{rad}},\infty)$}
\end{array} \right. \\
   &\cong
    \left\{ \begin{array}{ll}
 \eps^{1- \frac{\alpha_{\textup{rad}}}{\gamma}}, &
 \textrm{if $p\in(2,2_\star^{\textup{rad}})$,}\\
\eps^{1-\frac{1}{\gamma}} |\log(\eps)|^{-\frac{d-1}{d}}, & \textrm{if $p=2_\star^{\textup{rad}}$,}\\
\eps^{1-\frac{1}{\gamma}}, & \textrm{if $p\in(2_\star^{\textup{rad}},\infty)$.}
\end{array} \right.
\end{align*}

\medskip

\noindent{\bf $G_k$-symmetric case~\eqref{eq_LBReps}.}
Since $\textup{R}_\eps^k(p)=\textup{R}_\eps^{d-k}(p)$ it suffices to prove the claimed  
inequality  for $k\leq \lfloor d/2\rfloor$, so \eqref{eq_alphak} yields $\alpha_k=(d+k)(\frac12-\frac1p)$.
If $p>2_\star^k$, then $p>2_\star^{\textup{rad}}$ and thus
\[\textup R_\varepsilon^k(p)
\leq  \textup R_\varepsilon^{\textup{rad}}(p)
\lesssim  \eps^{1-\frac{1}{\gamma}},
\]
as claimed. Thus we may assume $2<p\leq2_\star^k$, which in particular forces $p<\frac{2k}{k-1}$.
So Lemma~\ref{lem_UpperBoundsPrep} applies and the choices $a_\eps(s)=(\eps+s^\gamma)^{-1}$ and
$m_{\eps,\delta}=\delta^2$ lead to
\begin{align*}
   \textup R_\varepsilon^k(p)
   \lesssim \frac{\int_0^{\delta^2} a_\eps^2(s)(\eps+s^\gamma)\d s \cdot
   \delta^{d-k}}{
   \left(\delta^{\frac{d-k}{p'}}
   (\delta^2)^{\frac{k-1}{2}-\frac{k}{p}} \int_0^{\delta^2}a_\eps(s)\d s\right)^2
   }
    \cong \delta^{2-2\alpha_k} \left(\int_0^{\delta^2} \frac{1}{\eps+s^\gamma} \d s \right)^{-1}.
\end{align*}
Setting   $\delta=\eps^{\frac1{2\gamma}}$
yields $\textup R_\varepsilon^k(p) \lesssim \eps^{1-\frac{\alpha_k}{\gamma}}$, which proves the claim.

\medskip

\noindent{\bf Non-symmetric case~\eqref{eq_LBReps2}.}
In this case, we have
$$
  \textup R_\varepsilon^\circ(p) \lesssim \textup R_\varepsilon^1(p) \cong
  \eps^{1-\frac{1\wedge \alpha_1}{\gamma}}. 
$$
This concludes the proof of Theorem \ref{thm_LUBounds}.
\end{proof}

\section{Fourier concentration}\label{sec_concentration}

In this section, we prove Theorem \ref{thm_concentration}.

\begin{lemma}\label{lem_conc1}
  Assume $d\geq 2$, $k\in\{1,\ldots,d-1\}$, $2< p< 2_s^\star$.
Let $g_\varepsilon$ be $(s,\gamma)$-admissible  for some $s>\frac{d}{d+1}$ and $\gamma>1$. 
  Let $\ast\in\{\circ,k,\textup{rad}\}$. Assume that $u_\eps$ is a minimizer for $\textup R^\ast_\eps(p)$
  and that there exist nonzero functions $v_\eps,w_\eps$, such that $u_\eps=v_\eps+w_\eps$ and $\widehat v_\eps\cdot\widehat
    w_\eps=0$ almost everywhere. Moreover, suppose that there exists $M_\eps>1$, such that
    \begin{equation}\label{eq_assumptionM}
        \frac{\textup q_\eps(w_\eps)}{\|w_\eps\|_p^2}\geq M_\eps \textup R_\eps^\ast(p).
    \end{equation}
    Then the following estimates hold:
    \begin{equation}\label{eq_estimatesM}
        \frac{\|w_\eps\|_p}{\|v_\eps\|_p}\leq\frac2{M_\eps-1},\text{ and }\,\,\,
    \frac{\textup q_\eps(w_\eps)}{\textup q_\eps(v_\eps)}\leq\frac4{M_\eps(1-M_\eps^{-1})^2}.
    \end{equation}
\end{lemma}
\begin{proof}
    The assumption $\widehat v_\eps\cdot\widehat w_\eps=0$ implies $\textup q_\eps(u_\eps)=\textup
    q_\eps(v_\eps)+\textup q_\eps(w_\eps)$. Minkowski's inequality and assumption \eqref{eq_assumptionM} then
    imply
    \[\textup R^\ast_\eps(p)=\frac{\textup q_\eps(u_\eps)}{\|u_\eps\|_p^2}\geq\frac{\textup
    q_\eps(v_\eps)+\textup q_\eps(w_\eps)}{(\|v_\eps\|_p+\|w_\eps\|_p)^2}\geq \textup
    R^\ast_\eps(p)\frac{\|v_\eps\|_p^2+M_\eps \|w_\eps\|_p^2}{(\|v_\eps\|_p+\|w_\eps\|_p)^2}.\]
    This can be rewritten as \[(\|v_\eps\|_p+\|w_\eps\|_p)^2\geq \|v_\eps\|_p^2+M_\eps \|w_\eps\|_p^2,\]
    which leads to the first estimate in \eqref{eq_estimatesM} after elementary manipulations. The second
    estimate in \eqref{eq_estimatesM} follows similarly:  \[\textup R^\ast_\eps(p)
        \geq\frac{\textup q_\eps(v_\eps)+\textup q_\eps(w_\eps)}{(\|v_\eps\|_p+\|w_\eps\|_p)^2}\geq
        \frac{\textup q_\eps(v_\eps)+\textup q_\eps(w_\eps)}{\left(\sqrt{\frac{\textup q_\eps(v_\eps)}{\textup
        R^\ast_\eps(p)}}+\sqrt{\frac{\textup q_\eps(w_\eps)}{M_\eps\textup R^\ast_\eps(p)}}\right)^2}\] can be
        rewritten as \[\left(\sqrt{\textup q_\eps(v_\eps)}+M_\eps^{-\frac12}\sqrt{\textup
        q_\eps(w_\eps)}\right)^2 \geq \textup q_\eps(v_\eps)+\textup q_\eps(w_\eps),\] 
        which leads to the second estimate in
    \eqref{eq_estimatesM}.
\end{proof}

Theorem \ref{thm_concentration} follows by verifying the lower bound \eqref{eq_assumptionM}
for some $M_\eps$ such that $M_\eps\to\infty$, as $\eps\to 0^+$.  The latter condition is the content of our next result.

\begin{lemma}\label{lem_LBMdelta}
Assume $d\geq 2$, $k\in\{1,\ldots,d-1\}$, $2< p< 2_s^\star$.
Let $g_\varepsilon$ be $(s,\gamma)$-admissible  for some $s>\frac{d}{d+1}$ and $\gamma>1$.
  Let $\ast\in\{\circ,k,\textup{rad}\}$.
  Let $u_\eps,v_\eps,w_\eps$ be as in Theorem~\ref{thm_concentration}, such that $\widehat w_\eps$
  vanishes identically on the spherical shell $A_{\eps,\delta_\eps}$ from \eqref{eq_Iinterval}--\eqref{eq_sphshell}
  for a given positive null sequence $\{\delta_\eps\}$.
 Then
    \begin{equation}\label{eq_LBM}
        \frac{\textup q_\eps(w_\eps)}{\|w_\eps\|_p^2}\gtrsim  M_\eps  \textup R_\varepsilon^\ast(p)
    \end{equation}
 holds for some $M_\eps= M_\eps(p,d,k,s,\gamma)>0$ satisfying $M_\eps\to\infty$, as
 $\eps\to 0^+$.
\end{lemma}
\begin{proof} 
  We only consider the case $\ast = k$ and  split the analysis into two cases.
    If $2_\star^k\leq p< 2_s^\star$, then reasoning as in \eqref{eq_UBsuperP} we obtain
    \begin{align*}
      \|w_\eps\|_p^2
      &\lesssim \textup q_\eps(w_\eps) \int_{|r-1|\notin I_{\eps,\delta_\eps}}\frac{\d r}{g_\eps(r)} \\
      &\lesssim \textup q_\eps(w_\eps) \left(\int_0^{\delta_\eps\eps^{1/\gamma}}\eps^{-1}\,\d s
    +\int_{\delta_\eps^{-1}\eps^{1/\gamma}}^\infty s^{-\gamma}\,\d s\right) \\
      &\cong \textup q_\eps(w_\eps) (\delta_\eps+ \delta_\eps^{\gamma-1}) \eps^{\frac{1}{\gamma}-1}.
    \end{align*}
    Here, we used $I_{\eps,\delta}=[\delta\eps^{1/\gamma},\delta^{-1}\eps^{1/\gamma}]$
    and that $g_\eps$ satisfies the lower bound in \eqref{eq_A2}.
    In light of \eqref{eq_LBReps}, it follows that \eqref{eq_LBM} holds with
    $M_\eps\cong(\delta_\eps+\delta_\eps^{\gamma-1})^{-1}$.
    If $2< p< 2_\star^k$, then reasoning as in \eqref{eq_L2BD}--\eqref{eq_UBvInterpol} yields
    \[
     \|w_\eps\|_p^2\leq (\|w_\eps\|_2^2)^{1-\alpha_k}
     (\|w_\eps\|_{2_\star^k}^2)^{\alpha_k}\lesssim(\eps^{-1}\textup
     q_\eps(w_\eps))^{1-\alpha_k}\left((\delta_\eps+\delta_\eps^{\gamma-1})\eps^{\frac1{\gamma}-1}\textup
     q_\eps(w_\eps)\right)^{\alpha_k},\] so \eqref{eq_LBReps} yields \eqref{eq_LBM} with $M_\eps
    \cong(\delta_\eps+\delta_\eps^{\gamma-1})^{-\alpha_k}$.
    Since  $2< p< 2_\star^k$ implies $\alpha_k>0$, this proves the claim.
\end{proof}

\section{Roughness}\label{sec_roughness}

In this section, we prove Theorem \ref{thm_roughness}, which should be regarded as a quantified version of
the fact that non-radial minimizers of the Rayleigh quotient asymptotically exhibit some sort of
Knapp-type behaviour.
In particular, their Fourier transforms must develop certain singularities along spheres which are near $\mathbb S^{d-1}$. 
To prove this, we use  Sobolev estimates for the adjoint
restriction operator \eqref{eq_ExtOp}. In the general (non-symmetric) case, optimal results were established by
Cho--Guo--Lee \cite{CGL15} but curiously their methods do not seem to allow for the required improved
estimates in the $G_k$-symmetric setting, which are key to  our analysis.

\begin{proposition}\label{prop_smoothness}
Let $d\geq 2$ and $k\in\{1,\ldots,d-1\}$. There exists $t^\star = t^\star(d)>0$ such that, for all
$t\in [0,t^\star)$ and  $G_k$-symmetric functions
$f\in H^t(\mathbb{S}^{d-1})$, the following estimate holds:
 \begin{equation}\label{eq_smoothing}
\|\widehat{f\sigma}\|_{2_\star^k(t)}\lesssim_t \|f\|_{H^t(\mathbb{S}^{d-1})},
\quad\text{where } \frac{1 }{2_\star^k(t)} := \frac{1-t/t^\star}{2_\star^k} +  
 \frac{t/t^\star}{2_\star^{\textup{rad}}}. 
 \end{equation}
 Moreover, estimate \eqref{eq_smoothing} for $k\in\{1,d-1\}$ holds for every function $f\in H^t(\mathbb{S}^{d-1})$.
\end{proposition}
\begin{proof}
In view of \eqref{eq_GkSymST}, the estimate
\begin{equation}\label{eq_GkST2}
\|\widehat{f\sigma}\|_{2_\star^k}\lesssim \|f\|_{L^2(\mathbb S^{d-1})}
\end{equation}
holds for every $G_k$-symmetric function $f:\mathbb{S}^{d-1}\to\Co$.
On the other hand, it was proved in \cite[Prop.\ 1]{Ma19} that the pointwise decay estimate
\[
  |\widehat{f\sigma}(x)| \lesssim \|f\|_{C^m(\mathbb S^{d-1})} (1+|x|)^{\frac{1-d}2}
\]
holds for every $f\in C^m(\mathbb S^{d-1})$ with $m=\lfloor\frac{d-1}2\rfloor+1$.
By Sobolev embedding, we may choose $t^\star=t^\star(d)>0$ such that
$H^{t^\star}(\mathbb S^{d-1})$ embeds into $C^m(\mathbb{S}^{d-1})$. Then
\begin{equation}\label{eq_HtSt}
\|\widehat{f\sigma}\|_{2_\star^{\textup{rad}},\infty} 
\lesssim \|f\|_{C^m(\mathbb S^{d-1})} \|(1+|x|)^{\frac{1-d}2}\|_{2_\star^{\textup{rad}},\infty}
\lesssim \|f\|_{H^{t^\star}(\mathbb S^{d-1})},
\end{equation}
for every $f\in H^{t^\star}(\mathbb S^{d-1})$.
The desired conclusion follows from interpolating \eqref{eq_GkST2} and \eqref{eq_HtSt}.
Indeed, given $t\in(0,t^\star)$, define $\theta_t:= t/t^\star$,
so that real interpolation \cite[Ch.\ 3]{BL} implies
\[\|\widehat{f\sigma}\|_{2_\star^k(t)}
\lesssim \|\widehat{f\sigma}\|_{2_\star^k(t),2}
\cong \|\widehat{f\sigma}\|_{(L^{2_\star^k},L^{2_{\star}^{\textup{rad}},\infty})_{\theta_t,2}}
\lesssim \|f\|_{(L^2(\mathbb S^{d-1}),H^{t^\star}(\mathbb S^{d-1}))_{\theta_t,2}}
\cong \|f\|_{H^t(\mathbb S^{d-1})}.
\]
In the last equality, we used  the identity  
\[(L^2(\mathbb S^{d-1}),H^{t^\star}(\mathbb S^{d-1}))_{\theta_t,2}
= H^t(\mathbb S^{d-1}),\]
which holds by definition of $\theta_t$ according to \cite[Theorem 6.2.4 and Theorem 6.4.4]{BL}.
\end{proof}

\begin{proof}[Proof of Theorem \ref{thm_roughness}]
  We focus on the proof for $G_k$-symmetric functions.  
  In view of Proposition \ref{prop_smoothness}, the non-symmetric case $\ast=\circ$
  is treated as the case $k\in\{1,d-1\}$. 
  
  Assume $2<p<2_\star^k$. 
  Lowering $t>0$ if necessary, we lose no generality in further assuming $2<p<2_\star^k(t)$ since
    $2_\star^k(t)\nearrow 2_\star^k$, as $t\to 0^+$.
    Aiming at a contradiction, suppose that \eqref{eq_ConclThmRough} does not hold, i.e., there exists $C<\infty$, such that
    \begin{equation}\label{eq_twdCtrdct}
        \sup_{|r-1|\in I_{\varepsilon,\delta_\varepsilon}} \frac{\|\widehat u_\varepsilon(r\cdot)\|_{H^t(\mathbb S^{d-1})}}{\|\widehat u_\varepsilon(r\cdot)\|_{L^2(\mathbb S^{d-1})}}\leq C,\text{ as } \varepsilon\to 0^+.
    \end{equation}
    Our aim is to use \eqref{eq_twdCtrdct} in order to verify the hypothesis~\eqref{eq_assumptionM}
    in Lemma~\ref{lem_conc1} with $M_\eps\to\infty$, as $\eps\to 0^+$. To this end, split $u_\eps=v_\eps+w_\eps$ as
   in \eqref{eq_split}, so that $\widehat w_\eps$ is supported outside the spherical shell $A_{\eps,\delta_\eps}$
   where $\delta_\eps\to 0^+$ as $\eps\to 0^+$. In Lemma \ref{lem_LBMdelta}, we have already verified 
   that this part is asymptotically negligible. More precisely, 
   \begin{equation}\label{eq_LBque2}
       \frac{\textup q_\eps(w_\eps)}{\|w_\eps\|_p^2}\gtrsim M_{\eps} \textup R_\eps^k(p),
   \end{equation}
   with $M_{\eps}\to\infty$, as $\eps\to 0^+.$  Next we use~\eqref{eq_twdCtrdct} in order to prove lower bounds  for the
   corresponding Rayleigh quotient involving $v_\eps$.
   To this end, we invoke Fourier inversion, the Sobolev estimate \eqref{eq_smoothing},
   assumption~\eqref{eq_twdCtrdct} and Lemma~\ref{lem_technical},
   yielding
   \begin{align*}
       \|v_\eps\|_{2_\star^k(t)}^2
       &\lesssim
       \left(\int_{|r-1|\in I_{\eps,\delta_\eps}} \left\|\int_{\mathbb S^{d-1}} e^{i\langle\omega,\cdot\rangle} \widehat v_\eps(r\omega)\,\d\sigma(\omega)\right\|_{2_\star^k(t)}\,\d r\right)^2\\
       &\lesssim\left( \int_{|r-1|\in I_{\eps,\delta_\eps}} \|\widehat v_\eps(r\cdot)\|_{H^t(\mathbb S^{d-1})}\,\d r\right)^2
        \leq C^2\left( \int_{|r-1|\in I_{\eps,\delta_\eps}} \|\widehat
       v_\eps(r\cdot)\|_{L^2(\mathbb S^{d-1})}\,\d r\right)^2\\
       &\lesssim \textup q_\eps(v_\eps)\left(\int_{|r-1|\in I_{\eps,\delta_\eps}}\frac{\d r}{g_\eps(r)}\right)
       \lesssim \eps^{\frac1{\gamma}-1} \textup q_\eps(v_\eps).
   \end{align*}
The key observation is that the exponent $2_\star^k(t)$ is strictly smaller than $2_\star^k$, and therefore the corresponding interpolation parameter
   \[\alpha_k(t):=\frac{\frac12-\frac1p}{\frac12-\frac1{2_\star^k(t)}}\in (0,1)\]
   satisfies $\alpha_k(t)>\alpha_k$; recall  \eqref{eq_alphak}.
   For $2<p<2_\star^k(t)$, we then have
   \begin{align*}
       \|v_\eps\|_p^2
       &\leq  (\|v_\eps\|_2^2)^{1-\alpha_k(t)}(\|u_\eps\|_{2_\star^k(t)}^2)^{\alpha_k(t)}\\
   &\lesssim (\eps^{-1} \textup q_\eps(v_\eps))^{1-\alpha_k(t)} (\eps^{\frac1{\gamma}-1}
   \textup q_\eps(v_\eps))^{\alpha_k(t)} = \eps^{\frac{\alpha_k(t)}{\gamma}-1}\textup q_\eps(v_\eps).
   \end{align*}
   In view of \eqref{eq_LBReps}, it follows that
   \begin{equation}\label{eq_LBque1}
       \frac{\textup q_\eps(v_\eps)}{\|v_\eps\|_p^2}\gtrsim \eps^{1-\frac{\alpha_k(t)}{\gamma}}
   \cong  \eps^{\frac{\alpha_k-\alpha_k(t)}{\gamma}} \textup R_\eps^k(p).
   \end{equation}
   From \eqref{eq_LBque2} and \eqref{eq_LBque1}, we then conclude
   \[\textup R_\eps^k(p)=\frac{\textup q_\eps(u_\eps)}{\|u_\eps\|_p^2}\geq\frac12\left(\frac{\textup q_\eps(v_\eps)}{\|v_\eps\|_p^2}\wedge\frac{\textup q_\eps(w_\eps)}{\|w_\eps\|_p^2}\right)
   \gtrsim\left(\eps^{\frac{\alpha_k-\alpha_k(t)}{\gamma}}\wedge M_{\eps}\right)\textup
   R_\eps^k(p),\] which is absurd since $\alpha_k-\alpha_k(t)<0$ and $M_{\eps}\to\infty$, as
   $\eps\to 0^+.$ 
   So the assumed uniform bound \eqref{eq_twdCtrdct} cannot hold, and this completes the proof of the theorem.
\end{proof}


\section*{Acknowledgements}
DOS\@ is partially supported by Centro de An\'{a}lise Matem\'{a}tica, Geometria e Sistemas Din\^{a}micos (CAMGSD) and
IST Santander Start Up Funds, and is grateful to Giuseppe Negro for valuable discussions during the preparation of this work.
This work was initiated during a pleasant visit of RM to Instituto Superior T\'{e}cnico, whose hospitality is
greatly appreciated.


\appendix
\section{Existence of ground states}\label{app_existence}

\begin{proposition}\label{prop_appexistence}
  Assume  $d\geq 2$ and $2<p<2_s^\star$. Given $\eps>0$,  let $g_\eps$ be $(s,\gamma)$-admissible
   for some  $s>\frac d{d+1}$ and $\gamma>1$.
 Then the infimum defining
 \[ 
  \textup R^\ast_\eps(p) \text{ for }\ast\in\{\circ,\textup{rad}\}
 \quad\text{or}\quad
 \textup R^k_\eps(p) \text{ for }k\in\{2,\ldots,d-2\}
 \]
  is attained. Any  minimizer is a non-trivial weak solution of \eqref{eq_BiHarNLSGen} after multiplication by a
  nonzero constant.
\end{proposition}
 \begin{proof} 
  The non-symmetric case of $\textup R^\circ_\eps(p)$ is covered by \cite[Theorem
  1]{BLSS20}. The discussion of the radially symmetric case $\ast = \textup{rad}$ is almost identical to the
  $G_k$-symmetric case for $k\in\{2,\ldots,d-2\}$, so we only discuss the latter. 
  
  Let $\{u_n\}\subset H_k^s$ be a minimizing sequence such that $\|u_n\|_p=1$ for each $n$, so that
  \[\textup q_\eps(u_n)\searrow \textup R_\eps^k(p), \text{ as }n\to\infty.\]
  Then the   sequence $\{\textup q_\eps(u_n)\}$ is bounded in $\R$, and so $\{u_n\}$ is likewise bounded in
  $H^s_k$. The latter statement is a consequence of the definition of
  $\textup q_\eps$ and the inequalities \eqref{eq_A1}--\eqref{eq_A2} for
  the $(s,\gamma)$-admissible function $g_\eps$. By Alaoglu's theorem, there exists a subsequence
  $\{u_{n_\ell}\}\subset H^s_k$ satisfying $u_{n_\ell}\rightharpoonup u$ in $H^s_k$, as $\ell\to\infty$.
  Since $k\in\{2,\ldots,d-2\}$, $H^s_k$ compactly embeds into $L^p(\R^d)$; 
  see~
  \cite[Th\'{e}or\`{e}me III.3]{L82}.
  As a consequence, $u_{n_\ell}\to u$ in $L^p(\R^d)$,  as $\ell\to\infty$, and in particular $\|u\|_p=1$.
  Being  continuous and
  convex, the functional $\textup q_\eps$ is weakly lower semicontinuous on $H^s_k$.
   Hence,
   \[\textup R_\eps^k(p)=\lim_{\ell\to\infty}\textup q_\eps(u_{n_\ell})\geq \textup q_\eps(u)=\frac{\textup
   q_\eps(u)}{\|u\|_p^2},\]
   and it follows that $u$ is a minimizer for $\textup R_\eps^k(p)$. Lagrange's multiplier
   rule implies that any minimizer $u$ is a weak solution after multiplication by a nonzero constant.
\end{proof}

\section{A real interpolation bound}\label{app_interpolation}

\begin{proposition}\label{prop:interpolate}
Let $0<r<q<\infty$ and $u:\R^d\to\Co$ be a measurable function such that $\|u\|_r\leq C_1$ and
$|u(x)|\leq C_2(1+|x|)^{-\frac dq}$ for almost every $x\in\R^d$. Then \[\|u\|_q \lesssim_{r,q,d} C_2 \left(1+\log_+\left(\frac{C_1}{C_2}\right)\right)^{\frac1q},\]
where $\log_+:=\log\vee \,{\bf 0}$.
\end{proposition}
\begin{proof}
We use the layer cake representation \cite[\S~1.13]{LL01} for $|u|$, set
$\lambda(t):=|\{x\in\R^d: |u(x)|\geq t\}|$.
The pointwise assumption implies 
$$
  \lambda(t)\lesssim_d (C_2t^{-1})^q \text{ for every }t>0, 
  \qquad 
  \lambda(t)=0 \text{ if }t\geq C_2.
$$
For any $\delta\in[0,C_2]$, we thus have that
\begin{align*}
\|u\|_q^q
&=q\int_0^{C_2} t^{q-1} \lambda(t)\,\textup d t
\lesssim_d \delta^{q-r} \int_0^\delta t^{r-1} \lambda(t)\,\textup d t+\int_\delta^{C_2} t^{q-1} (C_2t^{-1})^q\,\textup d t\\
&\lesssim \delta^{q-r} \|u\|_r^r + C_2^q\int_\delta^{C_2}
t^{-1}{\d t} 
\leq C_1^r \delta^{q-r}+C_2^q \log(C_2\delta^{-1}).
\end{align*}
Minimizing the latter quantity with respect to $\delta\in[0,C_2]$ leads to the choice
\[ \delta=
 \left\{ \begin{array}{ll}
((q-r)C_1^r  C_2^{-q})^{\frac1{r-q}}, & \textrm{if } C_1C_2^{-1}\geq(q-r)^{-\frac1r},\\
C_2, & \textrm{if }C_1C_2^{-1}<(q-r)^{-\frac1r}.
\end{array} \right. \]
In the first case this yields
\begin{align*}
\|u\|_q^q
&\lesssim_d\frac{C_2^q}{q-r}+C_2^q\left(\log(C_2)+\frac{\log(q-r)}{q-r}+\frac r{q-r}\log(C_1)-\frac q{q-r}\log(C_2)\right)\\
&=C_2^q\left(\frac{1+\log(q-r)}{q-r}-\frac{r}{q-r}\log(C_2)+\frac{r}{q-r}\log(C_1)\right)\\
&\lesssim_{r,q} C_2^q\left(1+\log_+\left(\frac{C_1}{C_2}\right)\right),
\end{align*}
and in the second case we even obtain the  better  bound $\|u\|_q \lesssim C_2$. 
\end{proof}

\section{ The case $k\in\{1,d-1\}$} \label{app_exceptionalcases}

We now explain in more detail why the symmetry groups $G_k=O(d-k)\times O(k)$ must be treated differently
when $k\in\{1,d-1\}$. We focus on $k=1$.
The first observation is that the optimal Stein--Tomas exponent for $G_1$-symmetric
functions is identical to the optimal exponent for general functions. In other words, 
$$
  2_\star^1 = 2_\star = 2\frac{d+1}{d-1},
$$ 
as observed in \cite[Remark 6.1]{MOS21}. As a consequence,
our methods do not allow to prove the strict inequality $\textup R_\eps^\circ(p)<\textup R^1_\eps(p)$ for
sufficiently small $\eps>0$ and any $p>2$, and it is an open question whether this holds or not. 
A related open question is whether some or even all ground states (i.e., minimizers for $\textup
R_\eps^\circ$) are $G_1$-symmetric up to a change of coordinates or whether $G_1$-symmetric solutions exist
at all. In fact, our existence proof of Proposition~\ref{prop_appexistence}
does not carry over to $k=1$ due to the failure of $H^s_1$ compactly embedding into $L^p(\R^d)$. 
To see this, take any nonzero test function $\varphi\in C_0^\infty(\R^2)$ and define
$u_n(x) := \varphi(|x'|^2,|x^d-n|^2)+\varphi(|x'|^2,|x^d+n|^2)$; here, $x=(x',x^d)\in\R^{d-1}\times\R$. Then $\{u_n\}$ is bounded in $H^s_1$,
but does not admit any $L^p(\R^d)$-convergent subsequence, so a compact embedding cannot exist. More
generally, it is unclear whether minimizing sequences for the Rayleigh quotient $\textup R_\eps^1(p)$ behave according to the ``dichotomy'' or the
``compactness'' alternative in Lions' concentration-compactness principle. 
In view of \cite[Lemma~2.4]{SEC13}, ``vanishing'' does not occur in the
Sobolev-subcritical regime $2<p<2_s^\star$.

\medskip

Nevertheless, when $d\geq 3$ we can prove the existence of solutions with the slightly smaller
symmetry group $G:= O(d-1)\times \{1\} \subset O(d)$, where  the evenness requirement with respect to the
last coordinate is dropped. Denote the corresponding Sobolev space by $H^s_G$. 
We first establish the following uniform decay estimate. 
 ~
\begin{proposition} \label{prop:radialdecay}
  Let $d\geq 2, q>2$, and $t>0$. Then there exists $\tau=\tau(d,q,t)>0$ such that, for all radially symmetric functions
  $u\in H^t(\R^d)$, the following holds for all $R>0$:
  $$ 
    \|u\|_{L^q(\R^d\setminus B_R(0))}   
    \lesssim  R^{-\tau} \|u\|_{H^t(\R^{d})}.
  $$ 
\end{proposition}
\begin{proof}
Fix $\eps>0$. By~\cite[Theorem~3.1]{DN18}, there exists $\vartheta\in (0,1)$ such that\footnote{Here, $[u]^2_{H^s}:= \int_{\R^d} |\widehat{u} (\xi)|^2 |\xi|^{2s}\,\d \xi.$} 
$$|u(x)|    \lesssim |x|^{\frac{1-d}{2}} [u]_{H^{\frac{1}{2}+\eps}}^\vartheta  \|u\|_2^{1-\vartheta}    \lesssim |x|^{\frac{1-d}{2}} \|u\|_{H^{\frac{1}{2}+\eps}}. 
  $$
  This implies, for $\frac{2d}{d-1}<\tilde q\leq\infty$,
  $$
    \|u\|_{L^{\tilde q}(\R^d\setminus B_R(0))} 
    \lesssim R^{\frac{d}{\tilde q}+\frac{1-d}{2}}  \|u\|_{H^{\frac{1}{2}+\eps}}.
$$
  We interpolate this estimate with the trivial bound    
  $\|u\|_{L^2(\R^d\setminus B_R(0))}  \leq \|u\|_2 = \|u\|_{H^0}$, and obtain
  $$
    \|u\|_{L^q(\R^d\setminus B_R(0))} 
    \lesssim R^{\theta(\frac{d}{\tilde q}+\frac{1-d}{2})} 
      \|u\|_{H^{\theta(\frac{1}{2}+\eps)}},
      \quad\text{if } \frac{\theta}{\tilde q}+\frac{1-\theta}{2} = \frac{1}{q},\; \theta\in [0,1].
  $$ 
  The freedom to arbitrarily choose $\theta\in (0,1],\tilde q>\frac{2d}{d-1}$ and $\eps>0$  allows us to
  conclude.
\end{proof}
~
\begin{theorem}\label{thm:ExistenceGSk1}
  Assume $d\geq 3$, $2<p<2_\star^s$, and $\eps>0$. 
Let $g_\eps$ be $(s,\gamma)$-admissible
   for some $s>\frac d{d+1}$ and $\gamma>1$. Then the infimum defining the Rayleigh quotient 
  \begin{equation}\label{eq_newRayleigh}
        \inf_{{\bf 0}\neq u\in H^s_G} \frac{\textup q_\eps(u)}{\|u\|_p^2}
  \end{equation}
  is attained. Any minimizer is a $G$-symmetric non-trivial weak solution of \eqref{eq_BiHarNLSGen} after
  multiplication by a nonzero constant. 
\end{theorem}

\begin{proof}
  The argument mimicks the proof of \cite[Theorem 1]{BLSS20},
  so we concentrate on the novel aspects.
  As in \cite{BLSS20}, one shows that there exist a minimizing sequence $\{u_n\}\subset H^s_G$, a sequence
  $\{x_n\}\subset\R^d$ and a nonzero function $u\in H^s(\R^d)$, such that $\|u_n\|_p=1$ and
  $u_n(\cdot+x_n)\rightharpoonup u$, as $n\to\infty$; see  \cite[Eq.\@ (2.2)]{BLSS20}.
 Writing $x_n=(x_n',x_n^d)\in\R^{d-1}\times\R$, we shall prove below  that the sequence $\{x_n'\}$
  is bounded in $\R^{d-1}$ and hence convergent to some limit $x_\infty'\in\R^{d-1}$, possibly after extraction of  a
  subsequence. We then obtain that $\tilde u_n(x) := u_n(x',x^d+x_n^d)$ is a
  minimizing sequence in $H^s_G$ that weakly converges to the nonzero function $\tilde u\in H^s_G$
  given by $\tilde u(x):=u(x'-x_\infty',x^d)$. 
  Here we used that $H^s_G$ is weakly closed in $H^s(\R^d)$. The argument can then be completed as
  in~\cite{BLSS20}, yielding $\tilde u$ as a minimizer for the Rayleigh quotient \eqref{eq_newRayleigh} and
  establishing the claim.

  \medskip

  It remains to prove that the sequence $\{x_n'\}\subset \R^{d-1}$ is bounded. Suppose not. Choose a test function
  $\varphi:\R^d\to\R$ with support in $K:=\overline{B_R(0)}\times [-R,R]\subset\R^d$ such that $\mu:= \int_{K}
  u\varphi \neq 0$; here, $B_R(0)$ denotes the ball in $\R^{d-1}$ with radius $R$. Then
  \begin{align*}
    \mu 
    &= \int_{K} u(x)\varphi(x)\d x \\
    &= \lim_{n\to\infty} \int_{K} u_n(x+x_n) \varphi(x)\d x \\
    &= \lim_{n\to\infty} \int_{K+(x_n',0)} u_n(x',x^d+x_n^d)  \varphi(x'-x_n',x^d) \d x'\d x^d \\
    &= \lim_{n\to\infty} \int_{-R}^{R} 
    \left(\int_{B_R(0)+x_n'} u_n(x',x^d+x_n^d)  \varphi(x'-x_n',x^d) \d x'\right)\d x^d.        
  \end{align*}
  We now apply Proposition~\ref{prop:radialdecay} to the radially symmetric functions 
  $$
    \R^{d-1}\ni x'\mapsto u_n(x',x^d+x_n^d)\in\R
  $$ 
  with $t\in (0,s-\frac{1}{2})$ and spatial dimension $d-1\geq 2$. This and the
  fractional Trace Theorem, i.e., the uniform boundedness of the trace operator $H^s(\R^d)\ni u\mapsto
  u|_{\R^{d-1}\times\{z\}}\in H^t(\R^{d-1})$ with respect to $z\in\R$, implies, for some $q>2$ and
  $\tau>0$,
  \begin{align*}
    \mu 
    &\leq \liminf_{n\to\infty} \int_{-R}^{R} 
    \left( \|u_n(\cdot,x^d+x_n^d)\|_{L^q(B_R(0)+x_n')}  \|\varphi(\cdot,x^d)\|_{L^{q'}(B_R(0)+x_n')} 
    \right)\d x^d \\
    &\lesssim_R \liminf_{n\to\infty}  \|\varphi\|_{\infty} 
    \int_{-R}^{R}  \|u_n(\cdot,x^d+x_n^d)\|_{L^q(\R^{d-1}\setminus B_{|x_n'|-R}(0))}   \d x^d \\
    &\lesssim_R \liminf_{n\to\infty}  \|\varphi\|_{\infty} 
    \int_{-R}^{R} (|x_n'|-R)^{-\tau} \|u_n(\cdot,x^d+x_n^d)\|_{H^t(\R^{d-1})}   \d x^d \\
    &\lesssim_R \liminf_{n\to\infty}  \|\varphi\|_{\infty} 
    (|x_n'|-R)^{-\tau} \|u_n\|_{H^s}. 
  \end{align*}
  The last term is zero because $\{u_n\}$ is bounded in $H^s_G$ and $|x_n'|\to \infty$ by assumption.
  This contradicts $\mu\neq 0$, so $\{x_n'\}\subset\R^{d-1}$ must be bounded.
\end{proof}

The preceding proof does not work for $G_1$-symmetric functions since {\it a priori} the $\tilde u_n$ are
only $G$-symmetric. In other words, the $\tilde u_n$
might not be even in the last coordinate. But if $u$ denotes a $G$-symmetric
minimizer/solution, then its reflection with respect to $x^d$ is another minimizer/solution. As a consequence, we find
either one $G_1$-symmetric minimizer or two distinct $G$-symmetric minimizers which are
related to each other by a reflection. Given Theorem~\ref{thm:ExistenceGSk1} and Theorem~\ref{thm_gen}, one may
ask whether ground states are $G$-symmetric or not. 
Furthermore, as far as the asymptotic regime  is concerned, one may check that
the Fourier transform of any $G$-symmetric minimizer concentrates on the unit sphere and
becomes rough  as $\eps\to 0^+$; the proof is identical to the $G_k$-symmetric case.



\begin{thebibliography}{03}

\bibitem{BL}
\textsc{J. Bergh, J. L\"ofstr\"om},
\newblock Interpolation spaces. An introduction.
\newblock Grundlehren der Mathematischen Wissenschaften, No. 223. Springer-Verlag, Berlin-New York, 1976.

\bibitem{BLSS20}
\textsc{L. Bugiera, E. Lenzmann, A. Schikorra, J. Sok},
\newblock \textit{On symmetry of traveling solitary waves for dispersion generalized NLS.}
\newblock Nonlinearity \textbf{33} (2020), no.~6, 2797--2819.

\bibitem{CGL15}
\textsc{Y. Cho, Z. Guo, S. Lee},
\newblock \textit{A Sobolev estimate for the adjoint restriction operator.}
\newblock Math. Ann. \textbf{362} (2015), no.~3--4, 799--815.

\bibitem{DN18}
\textsc{P. L. De N\'{a}poli},
\newblock \textit{Symmetry breaking for an elliptic equation involving the fractional {L}aplacian.}
\newblock  Differential Integral Equations \textbf{31} (2018), no~1--2, 75--94.

\bibitem{Gr14}
\textsc{L. Grafakos},
\newblock Classical Fourier analysis.
\newblock Third edition. Graduate Texts in Mathematics, 249. Springer, New York, 2014.

\bibitem{LS2021}
\textsc{E. Lenzmann,  J. Sok},
\newblock \textit{A sharp rearrangement principle in {F}ourier space and
              symmetry results for {PDE}s with arbitrary order},
\newblock Int. Math. Res. Not. IMRN   (2021), no.~19, 15040--15081.

\bibitem{LW21}
\textsc{E. Lenzmann, T. Weth},
\newblock \textit{Symmetry breaking for ground states of biharmonic NLS via Fourier extension estimates.}
\newblock arXiv:2110.10782. 

\bibitem{LL01}
\textsc{E. H. Lieb, M. Loss},
\newblock Analysis.
\newblock Second edition. Graduate Studies in Mathematics, 14. American Mathematical Society, Providence, RI, 2001.


\bibitem{LM72}
\textsc{J.-L. Lions, E. Magenes},
\newblock Non-homogeneous boundary value problems and applications. Vol. I.
\newblock Translated from the French by P. Kenneth.
\newblock Die Grundlehren der mathematischen Wissenschaften, Band 181. Springer-Verlag, New York-Heidelberg, 1972.

	\bibitem{L82}
\textsc{P.-L. Lions},
\newblock \textit{Sym\'{e}trie et compacit\'{e} dans les espaces de {S}obolev},
\newblock J. Functional Analysis \textbf{49} (1982), no.~3, 315--334.	

\bibitem{Ma19}
\textsc{R. Mandel},
\newblock \textit{Uncountably many solutions for nonlinear Helmholtz and curl-curl equations.}
\newblock Adv. Nonlinear Stud. \textbf{19} (2019), no.~3, 569--593.

\bibitem{MOS21}
\textsc{R. Mandel, D. Oliveira e Silva},
\newblock \textit{The Stein--Tomas inequality under the effect of symmetries.}
\newblock arXiv:2106.08255. J. Anal. Math., to appear. 

\bibitem{SEC13}
\textsc{S. Secchi},
\newblock \textit{Ground state solutions for nonlinear fractional Schr\"odinger equations in $\R^N$.}
\newblock J. Math. Phys. \textbf{54}, 031501 (2013).  

\bibitem{St93}
\textsc{E. M. Stein},
\newblock Harmonic analysis:\@ real-variable methods, orthogonality, and oscillatory integrals.
\newblock Princeton Mathematical Series, 43. Monographs in Harmonic Analysis, III. Princeton University Press, Princeton, NJ, 1993.

\bibitem{To75}
\textsc{P. Tomas},
\newblock\textit{A restriction theorem for the Fourier transform.}
\newblock Bull.\@ Amer.\@ Math.\@ Soc.\@ \textbf{81} (1975), 477--478.

\end{thebibliography}
\end{document}